\documentclass{article}
\usepackage[utf8]{inputenc}
\pdfoutput=1
\usepackage{amsmath,color,graphicx,amssymb}
\usepackage{amsthm}
\usepackage{amsmath}
\usepackage{tikz}
\usetikzlibrary{calc,matrix,arrows,shapes.geometric,positioning,decorations.pathmorphing}

\DeclareMathOperator{\Tor}{Tor}

\newtheorem{thm}{Theorem}[section] 

\newtheorem{lemma}[thm]{Lemma}     

\newtheorem{cor}[thm]{Corollary}

\newtheorem{prop}[thm]{Proposition}

\theoremstyle{definition}
\newtheorem{defn}[thm]{Definition}

\theoremstyle{definition}
\newtheorem{ex}[thm]{Example}

\theoremstyle{definition}
\newtheorem{remk}[thm]{Remark}

\theoremstyle{definition}

\theoremstyle{definition}
\newtheorem{quest}[thm]{Question}

\allowdisplaybreaks

\title{\large{\uppercase{Generators of Koszul Homology with Coefficients in a} $\underline{g}$-\uppercase{Weak Complete Intersection Module}}}
\author{\normalsize{\uppercase{Rachel N. Diethorn}}}
\date{}

\begin{document}

\maketitle

\begin{abstract}
We discuss a class of modules, which we call $\underline{g}$-weak complete intersection modules, inspired by the weak complete intersection ideals studied by Rahmati, Striuli, and Yang and we present explicit formulas for the generators of Koszul homology with coefficients in a $\underline{g}$-weak complete intersection module.  This generalizes work of Herzog and of Corso, Goto, Huneke, Polini, and Ulrich.  We use these explicit formulas to study connections between $\underline{g}$-weak complete intersection ideals and weak complete intersection ideals.
\end{abstract}

\section{Introduction}
\paragraph
\indent  Let $Q$ be a commutative, Noetherian $k$-algebra where $k$ is a field of characteristic zero and let $\underline{g}=g_1,...,g_s$ be a regular sequence in $Q$.  In this paper, we study a class of ideals which we call $\underline{g}$-weak complete intersection ideals, inspired by the weak complete intersection ideals studied by Rahmati, Striuli, and Yang in their recent paper \cite{RAHMATI2018129}.  Weak complete intersection ideals in a local ring are the ideals $I$ such that the differentials in the minimal free resolution $F$ of the quotient $Q/I$ land in $IF$.  Given a regular sequence $\underline{g}$ in a local ring $Q$, we define $\underline{g}$-weak complete intersection ideals to be the ideals $I$ such that the differentials in the minimal free resolution of $Q/I$ land in $(\underline{g})F$.  We define the more general notion of a $\underline{g}$-weak complete intersection module in a similar way.   Expanding the allowed range for the differentials allows additional flexibility, which as we show in this paper, becomes useful in studying weak complete intersection ideals.   Of course, when Q is regular and $\underline{x}$ is a regular system of parameters (that is, it minimally generates the maximal ideal), then any ideal is an $\underline{x}$-weak complete intersection ideal, but, even under less restrictive hypotheses, interesting examples are abundant.  We observe that for a given regular sequence $\underline{g}$, the intersection of the class of $\underline{g}$-weak complete intersection ideals and the class of weak complete intersection ideals lies inside the the class of complete intersection ideals.  Although neither class of ideals is contained in the other, we find other connections between them.  To accomplish this, we study the Koszul homology $H(\underline{g};R)$ where $R=Q/I$ is the quotient by a $\underline{g}$-weak complete intersection ideal.  
\newline
\indent One approach to understanding Koszul homology is to find, explicitly, its generators.  In his 1991 paper \cite{MR1310371} (see also \cite{Herzog2018}), Herzog gives explicit formulas for generators of the homology $H(\underline{x};R)$ of the Koszul complex on the minimal generators of the of the irrelevant maximal ideal $m=(x_1,...,x_n)$ of a finitely generated graded $k$-algebra $R$, where $k$ is a field of characteristic zero. Recently, the authors of \cite{MR3717974} provided explicit formulas for generators of Koszul homology in a more general setting, namely they studied the homology of the Koszul complex on a sequence $x_1^{a_1},...,x_n^{a_n}$ with $M$-coefficients where $M$ is what we call a $\underline{g}$-weak complete intersection module throughout this paper, with $\underline{g}=x_1^{a_1},...,x_n^{a_n}$.  In this paper, we further generalize to the setting where $\underline{g}$ is any full regular sequence and $M$ is a $\underline{g}$-weak complete intersection module.  One of the main tools we use to obtain these explicit formulas is the classical Perturbation Lemma from homological algebra.  We also utilize theory developed in \cite{dyckerhoff2013} on the formulation of a sort of partial derivative with respect to our regular sequence $\underline{g}$ and the de Rham contraction built from it.  The formulas we provide are given in terms of these partials.     
\newline
\indent  The explicit formulas for generators of Koszul homology given in this paper become useful in studying the connections between $\underline{g}$-weak complete intersection ideals and weak complete intersection ideals, where $\underline{g}$ is any full regular sequence.  In particular, we study the question of when $(\underline{g})$ is a weak complete intersection ideal in the quotient by a $\underline{g}$-weak complete intersection ideal and we give a sufficient condition involving the partials with respect to $\underline{g}$ mentioned above.  This provides a new family of examples of weak complete intersections.   
\newline
\indent  We now outline the contents of this paper.  In Section 2, we introduce the notion of $\underline{g}$-weak complete intersection modules and ideals and give some examples.  In Proposition \ref{propfpd}, we show  that for a given regular sequence $\underline{g}$, the intersection of the class of $\underline{g}$-weak complete intersection ideals and the class of weak complete intersection ideals lies inside the the class of complete intersection ideals. In the remainder of Section 2, we introduce the main tools we use to obtain formulas for the generators Koszul homology, including the perturbation lemma and the formulation of the de Rham contraction from \cite{dyckerhoff2013}.  In Section 3, we provide explicit formulas for the generators of Koszul homology $H(\underline{g};M)$ where $M$ is a $\underline{g}$-weak complete intersection module.  The formulas can be found in Theorem \ref{mainthm} and Corollaries \ref{cor1} and \ref{cor2}.  In Section 4, we utilize the explicit formulas from the previous section to study the question of when $(\underline{g})$ is a weak complete intersection ideal in the quotient $R=Q/I$ where $I$ is a $\underline{g}$-weak complete intersection ideal.  We give a sufficient condition in Proposition \ref{propapp} and show it is not necessary in Example \ref{exconverse}.

\section{Preliminaries}
\paragraph
\indent  In this section, we introduce the notion of $\underline{g}$-weak complete intersection ideals and modules.  We then discuss the main tools used throughout the paper, including the Perturbation Lemma and a version of the de Rham contraction developed in \cite{dyckerhoff2013}.

\subsection{$\underline{g}$-Weak Complete Intersection Ideals and Modules}
\paragraph
\indent  In this section, we introduce a class of ideals which we call $\underline{g}$-weak complete intersection ideals, inspired by the weak complete intersection ideals defined in \cite{RAHMATI2018129}.  We also introduce the related notion of $\underline{g}$-weak complete intersection modules.  Let $Q$ be a commutative, Noetherian $k$-algebra and let $\underline{g}=g_1,...,g_s$ be a regular sequence in $Q$.   

\begin{defn}\label{defwci}
A finitely generated $Q$-module $M$ is a \textit{$\underline{g}$-weak complete intersection module} if there is a free resolution $(F,\partial)$ of $M$ which satisfies the property: Im$\,\partial_i\subseteq (\underline{g})F_{i-1}$ for every $i$.  We call an ideal $I\subseteq Q$ a \textit{$\underline{g}$-weak complete intersection ideal} if $Q/I$ is a $\underline{g}$-weak complete intersection module.   
\end{defn}

In the local case, we note that it suffices to consider the minimal free resolution.  We now give some examples of $\underline{g}$-weak complete intersection ideals.

\begin{ex}
    The complete intersection ideal $(\underline{g})\subseteq Q$ is a $\underline{g}$-weak complete intersection ideal.  Indeed, $Q/(\underline{g})$ is minimally resolved by the Koszul complex on $\underline{g}$, whose differentials certainly land in the ideal $(\underline{g})$.  More generally, any embedded complete intersection ideal $(\underline{f})\subseteq (\underline{g})$ is also a $\underline{g}$-weak complete intersection ideal for the same reason.  In a regular ring, every module is an $\underline{x}$-weak complete intersection, where $\underline{x}$ is a minimal set of generators for the maximal ideal. 
\end{ex}

The next example gives a large class of non-complete intersection ideals.

\begin{ex}
Let $Q$ be a polynomial ring and fix a regular sequence $\underline{g}$ in $Q$.  Then any ideal generated by monomials in $\underline{g}$  is a $\underline{g}$-weak complete intersection ideal.  Indeed, $Q/I$ is resolved (possibly non-minimally) by the Taylor resolution for monomials in a regular sequence.  The entries in the differentials are either monomials in $\underline{g}$ or units.  After change of bases, the minimal free resolution splits off.  The entries in the differentials of the minimal resolution are still contained in $\underline{g}$ as the appropriate row and column operations do not disturb this property. 
\end{ex}

The next example shows that a $\underline{g}$-weak complete intersection ideal need not be a monomial ideal. 

\begin{ex}\label{expowersvars}
Let $Q=k[x,y,z]$ be a polynomial ring and take $\underline{g}$ to be the regular sequence $\underline{g}=x^2,y^3,z^5$.  Let $I=(x^2y^4+y^3z^7, y^6, x^4y^2)$.  According to Macaulay2, a free resolution of $Q/I$ over $Q$ is given by
\begin{align*}
    0\rightarrow Q\overset{\partial_3}{\longrightarrow} Q^3\overset{\partial_2}{\longrightarrow} Q^3\overset{\partial_1}{\longrightarrow} Q\rightarrow Q/I\rightarrow 0
\end{align*}
where the differentials are given by the following matrices:
\begin{align*}
\partial_3=
\left[
\begin{array}{c}
-z^7-x^2y \\
-x^4 \\
y^3
\end{array}
\right],\thinspace\thinspace\thinspace
&\partial_2=
\left[
\begin{array}{c c c}
-y^4 & 0 & -yz^7-x^2y^2 \\
x^4 & -z^7-x^2y & 0 \\
0 & y^3 & x^4
\end{array}
\right],
\end{align*}
\begin{align*}
\partial_1=&
\left[
\begin{array}{c c c}
x^4y^2 & y^6 & x^2y^4+y^3z^7 \\
\end{array}
\right] 
\end{align*}
It is easy to see that $I$ is a $\underline{g}$-weak complete intersection ideal.
\end{ex}

Although the definition of $\underline{g}$-weak complete intersection ideals was inspired by the definition of weak complete intersection ideals, for a fixed regular sequence $\underline{g}$, the two classes of ideals are distinct and neither one is contained in the other.  In fact, the intersection of the two classes is contained in the class of complete intersection ideals, as shown in the following proposition.

\begin{prop}\label{propfpd}
Let $Q$ be a local (or graded) ring and let $\underline{g}$ be a (homogeneous) regular sequence in $Q$.  
\begin{itemize}
\item[(1)] Every finitely generated $\underline{g}$-weak complete intersection module has finite projective dimension. 
\item[(2)] A weak complete intersection ideal is a $\underline{g}$-weak complete intersection ideal if and only if it is a complete intersection ideal embedded in $(\underline{g})$. 
\end{itemize}  
\end{prop}

\begin{proof}  We give a proof of the local case; the graded case is similar. \newline
(1)  Let $M$ be a finitely generated $\underline{g}$-weak complete intersection module and let $F$ be its minimal free resolution over $Q$.  Then we have
\begin{align*}
H_{\ell}(\underline{g};M)&=H_{\ell}(K(\underline{g};Q)\otimes_Q M) \\
&=\Tor_{\ell}^Q(Q/(\underline{g}),M) \\
&=H_{\ell}(Q/(\underline{g})\otimes_Q F) \\
&=Q/(\underline{g})\otimes_Q F_{\ell}
\end{align*} 
where the second equality follows from the fact that $\underline{g}$ is a regular sequence and the last equality follows directly from Definition \ref{defwci}.  Note that $H_{\ell}(\underline{g};M)=0$ for all $\ell>s$, where $\underline{g}=g_1,...,g_s$.  Then $Q/(\underline{g})\otimes_Q F_{\ell}=0$ by above and hence $F_{\ell}=0$ for all $\ell>s$ by Nakayama's Lemma.  Thus, $\text{pd}_Q M<\infty$. \newline
(2)  Let $I$ be a weak complete intersection ideal and suppose it is also a $\underline{g}$-weak complete intersection ideal.  Note that the minimal generators of $I$ are contained in the ideal $(\underline{g})$ since they are the entries in the first differential of the minimal free resolution $F$ of $Q/I$ over $Q$.  And by (1), $\,\text{pd}_Q Q/I<\infty$, thus $I$ is a complete intersection ideal by \cite[Remark 2.4]{RAHMATI2018129}.  The other direction is clear.      
\end{proof}
 
In Section 3, we study the homology of the Koszul complex on $\underline{g}$ with coefficients in a $\underline{g}$-weak complete intersection module and in Section 4, we investigate further the connection between $\underline{g}$-weak complete intersection ideals and weak complete intersection ideals.

\subsection{The Perturbation Lemma}
\paragraph
\indent In this section, we discuss the classical perturbation lemma which we will use in Section 3 as the main tool for providing explicit formulas for the generators of the homology of the Koszul complex on $\underline{g}$ with coefficients in a $\underline{g}$-weak complete intersection module.  We begin with the relevant definitions, which can be found in \cite{2004math......3266C}.  

\begin{defn}
A \textit{deformation retract datum}
\begin{align*}
    \Big((F,\partial_F)\underset{i}{\overset{p}{\leftrightarrows}}(G,\partial_G), H\Big)
\end{align*}
consists of the following:
\begin{description}
    \item[(i)] complexes $(F,\partial_F)$ and $(G,\partial_G)$
    \item[(ii)] quasi-isomorphisms $p$ and $i$
    \item[(iii)]  a homotopy $H$ between $ip$ and $\text{Id}_G$ (ie.\thinspace\thinspace $\partial_G H+H\partial_G=ip-\text{Id}_G$)
\end{description}
such that $pi=\text{Id}_G$.  A \textit{special deformation retract datum} is a deformation retract datum which also satisfies $Hi=0$, $pH=0$, and $H^2=0$.
\end{defn}

Given a deformation retract datum, one can define a (small) perturbation of the datum as follows.

\begin{defn}
A \textit{perturbation} of a deformation retract datum is a map 
\begin{align*}
    G\overset{\epsilon}{\longrightarrow}G
\end{align*}
of the same degree as $\partial_G$, such that $(\partial_G+\epsilon)^2=0$.  The perturbation is \textit{small} if $\text{Id}_G-\epsilon H$ is an invertible map.
\end{defn}

Now we state the perturbation lemma; see for example \cite[2.4]{2004math......3266C}.

\begin{thm}
[\textbf{Perturbation Lemma}]  If $\epsilon$ is a small perturbation of the deformation retract datum
\begin{align*}
     \Big((F,\partial_F)\underset{i}{\overset{p}{\leftrightarrows}}(G,\partial_G), H\Big)
\end{align*}
then the perturbed datum 
\begin{align*}
     \Big((F,\widetilde{\partial_F})\underset{\tilde{i}}{\overset{\tilde{p}}{\leftrightarrows}}(G,\partial_G+\epsilon), \tilde{H}\Big)
\end{align*}
with 
\begin{align*}
    \widetilde{\partial_F}=\partial_F+pAi,\thinspace\thinspace\thinspace \tilde{p}=p+pAH,\thinspace\thinspace\thinspace \tilde{i}=i+HAi,\thinspace\thinspace\thinspace \tilde{H}=H+HAH
\end{align*}
where $A=(\text{Id}_G-\epsilon H)^{-1}\epsilon$, is a deformation retract datum.  In particular, $\tilde{i}$ is a homotopy equivalence.
\end{thm}

In Section 3, we modify a deformation retract datum constructed in \cite{dyckerhoff2013}, and apply the perturbation lemma to provide explicit formulas for generators of Koszul homology.

\subsection{The de Rham Contraction}
\paragraph
\indent In this section, we utilize theory developed in \cite{dyckerhoff2013} on connections to formulate a sort of partial derivative with respect to the elements of a regular sequence $\underline{g}$.  We use these partial derivatives to give explicit formulas for the generators of the homology of the Koszul complex on $\underline{g}$ with coefficients in a $\underline{g}$-weak complete intersection module.
\newline
\indent Throughout this section we assume that $k$ a field of characteristic zero and $Q$ is a Noetherian $k$-algebra.  We let $\underline{g}$ be a regular sequence in $Q$ such that $Q/(\underline{g})$ is a finite dimensional $k$-vector space and $Q$ is complete\footnote{As in \cite{MR1011461}, we use the term ``complete'' to mean ``complete and separated''} in the $(\underline{g})$-adic topology.  We note that $Q$ is an algebra over the polynomial ring $k[\underline{g}]$ via the inclusion map. 
\newline
\indent We let $\Omega^1_{k[\underline{g}]/k}$ be the module of Kahler differentials, namely
\begin{align*}
    \Omega^1_{k[\underline{g}]/k}=\bigoplus_{i=1}^s k[\underline{g}]dg_i,
\end{align*} 
and we denote by $\Omega_{k[\underline{g}]/k}$ the exterior algebra over $\Omega^1_{k[\underline{g}]/k}$ with differential induced by the Euler map $\Omega^1_{k[\underline{g}]/k}\rightarrow k[\underline{g}]$ sending $dg_i$ to $g_i$.  Then the Koszul complex on $\underline{g}$ over $Q$ is given by $Q\underset{k[\underline{g}]}{\otimes}\Omega_{k[\underline{g}]/k}$.  Indeed,
\begin{align*}
    Q\underset{k[\underline{g}]}{\otimes}\Omega_{k[\underline{g}]/k}&=Q\underset{k[\underline{g}]}{\otimes}\bigwedge\Omega^1_{k[\underline{g}]/k} \\
    &=Q\underset{k[\underline{g}]}{\otimes}k[\underline{g}]\langle dg_1,...,dg_s|\delta(dg_i)=g_i\rangle \\
    &=Q\langle dg_1,...,dg_s|\delta(dg_i)=g_i\rangle.
\end{align*}
Let $\pi:Q\rightarrow Q/(\underline{g})$ be the usual quotient map and fix a $k$-linear splitting 
\begin{align*}
    \sigma:Q/(\underline{g})\rightarrow Q
\end{align*}
of $\pi$.  The following lemma is a well known; see for example \cite[Appendix B]{dyckerhoff2013} or \cite[Lemma 3.1.1]{MR868864}.  As a proof is not given in either source above, we include a proof here to clarify the constructions in this section.

\begin{lemma}\label{lemmalipman}
For every element $q\in Q$, there exist unique residue classes $\bar{q_N}\in Q/(\underline{g})$ such that
\begin{align*}
    q=\sum_{N\in\mathbb{N}^s}\sigma(\bar{q_N})g^N.
\end{align*}
where $g^N=g_1^{N_1}...g_s^{N_s}$.
\end{lemma}

\begin{proof}
We begin by  noting that $\pi(q-\sigma(\bar{q}))=\bar{q}-\pi(\sigma(\bar{q}))=\bar{q}-\bar{q}=0$, where the second equality follows from the fact that $\sigma$ is a splitting and where $\bar{q}=\pi(q)$.  Hence, $q-\sigma(\bar{q})\in (\underline{g})$.
\newline
\indent Let $\pi_i\colon (\underline{g})^i\rightarrow (\underline{g})^i/(\underline{g})^{i+1}$ be the usual quotient maps.  Since $\underline{g}$ is a regular sequence, we have that $(\underline{g})^i/(\underline{g})^{i+1}$ is a free $Q/(\underline{g})$-module.  We now use induction on $i$.  We know that $\bar{g_1},...,\bar{g_s}$ is a basis for $(\underline{g})/(\underline{g})^{2}$ over $Q/(\underline{g})$.  Thus, we can write 
\begin{align*}
    \pi_1(q-\sigma(\bar{q}))&=\sum_{i=1}^s \bar{q_{i}}\bar{g_i}  \\
    &=\sum_{\substack{N=(n_1,...,n_s) \\ n_1+...+n_s=1}}\bar{q_N}\bar{g}^N
\end{align*}
for $\bar{q_{i}}\in Q/(\underline{g})$.
\newline
\indent Now we suppose that
\begin{align*}
    q-\Bigg(\sum_{\substack{N=(n_1,...,n_s) \\ n_1+...+n_s< m-1}}\sigma(\bar{q_N})g^N\Bigg)\in(\underline{g})^{m-1}
\end{align*}
and 
\begin{align*}
    \pi_{m-1}\Bigg(q-\Bigg(\sum_{\substack{N=(n_1,...,n_s) \\ n_1+...+n_s< m-1}}\sigma(\bar{q_N})g^N\Bigg)\Bigg)=\sum_{\substack{N=(n_1,...,n_s) \\ n_1+...+n_s=m-1}}\bar{q_N}g^N.
\end{align*}
We note that
\begin{align*}
    \pi_{m-1}\Bigg(\sum_{\substack{N=(n_1,...,n_s) \\ n_1+...+n_s= m-1}}\sigma(\bar{q_N})g^N\Bigg)&=\sum_{\substack{N=(n_1,...,n_s) \\ n_1+...+n_s= m-1}}\bar{q_N}\bar{g}^N \\
    &= \pi_{m-1}\Bigg(q-\Bigg(\sum_{\substack{N=(n_1,...,n_s) \\ n_1+...+n_s< m-1}}\sigma(\bar{q_N})g^N\Bigg)\Bigg)
\end{align*}
where the first equality follows from the fact that $\sigma$ is a splitting and the second follows from the inductive hypothesis.  Thus 
\begin{align*}
    q-\Bigg(\sum_{\substack{N=(n_1,...,n_s) \\ n_1+...+n_s\leq m-1}}\sigma(\bar{q_N})g^N\Bigg)\in(\underline{g})^m.
\end{align*}
Again, since $\underline{g}$ is a regular sequence, we have that $(\underline{g})^m/(\underline{g})^{m+1}$ is a free $Q/(\underline{g})$-module with basis $\{\bar{g_1}^{i_1}...\bar{g_s}^{i_s}\}_{i_1+...+i_s=m}$, so 
\begin{align*}
    \pi_m\Bigg(q-\Bigg(\sum_{\substack{N=(n_1,...,n_s) \\ n_1+...+n_s\leq m-1}}\sigma(\bar{q_N})g^N\Bigg)\Bigg)=\sum_{\substack{N=(n_1,...,n_s) \\ n_1+...+n_s= m}}\bar{q_N}\bar{g}^N.
\end{align*}
Thus by induction, we have 
\begin{align*}
    q-\Bigg(\sum_{N\in\mathbb{N}^s}\sigma(\bar{q_N})g^N\Bigg)\in\bigcap_{n\geq 0} (\underline{g})^n=(0)
\end{align*} 
since $Q$ is complete (and thus separated) in $(\underline{g})$-adic topology.  Thus,
\begin{align*}
    q=\sum_{N}\sigma(\bar{q_N})g^N
\end{align*}
as desired.  
\newline
\indent To show uniqueness, we suppose that 
\begin{align*}
    \sum_{N}\sigma(\bar{q_N})g^N=q=\sum_{N}\sigma(\bar{r_N})g^N.
\end{align*}
Then we have that 
\begin{align*}
    \sum_{N}\sigma(\overline{q_N-r_N})g^N=0
\end{align*}
and thus
\begin{align*}
    0=\pi\Bigg(\sum_{N}\sigma(\overline{q_N-r_N})g^N\Bigg)=\overline{q_0-r_0},
\end{align*}
giving us the equality $q_0=r_0$. 
\newline
\indent Now suppose that $q_{N}=r_{N}$ for all $N=(n_1,...,n_s)$ with $\sum_{i=1}^s n_i\leq M-1$.  To show that this equality holds for $N$ such that $\sum_{i=1}^s n_i=M$, we note that
\begin{align*}
   \sum_{\substack{N=(n_1,...,n_s) \\ n_1+...+n_s\geq M}}\sigma(\overline{q_N-r_N})g^N=0 
\end{align*}
and thus
\begin{align*}
   0&=\pi_M\Bigg(\sum_{\substack{N=(n_1,...,n_s) \\ n_1+...+n_s\geq M}}\sigma(\overline{q_N-r_N})g^N\Bigg) \\
   &=\sum_{\substack{N=(n_1,...,n_s) \\ n_1+...+n_s= M}}\overline{q_N-r_N}\bar{g}^N
\end{align*}
But $\{\bar{g_1}^{n_1}...\bar{g_s}^{n_s}\}_{n_1+...+n_s=M}$ is a basis for $(\underline{g})^M/(\underline{g})^{M+1}$, so we get that $\overline{q_N}=\overline{r_N}$ for $N=(n_1,...,n_s)$ with $\sum_{i=1}^s n_i=M$, which completes induction.
\end{proof}

By writing elements of $Q$ in this way, Dyckerhoff and Murfet in \cite{dyckerhoff2013} give an explicit $k$-linear connection on $Q$, namely
\begin{align*}
    \nabla^0\colon Q\rightarrow Q\underset{k[\underline{g}]}{\otimes}\Omega^1_{k[\underline{g}]/k}
\end{align*}
given by
\begin{align*}
    \nabla^0(q)=\sum_{i=1}^s\sum_{N}N_i\sigma(\overline{q_N})g^{N-e_i}\otimes dg_i
\end{align*}
where $e_i\in\mathbb{Z}^s$ are the standard basis vectors and where for a tuple $N$ involving negative integers, $g^N$ is defined to be zero.  By connection we mean a $k$-linear map satisfying the Leibniz rule.  By means of this connection, one can define $\frac{\partial}{\partial g_i}$ to be the $k$-linear map given by the composition
\begin{align*}
    \frac{\partial}{\partial g_i}\colon Q\overset{\nabla^0}{\longrightarrow}Q\underset{k[\underline{g}]}{\otimes}\Omega^1_{k[\underline{g}]/k}\overset{(dg_i)^*}{\longrightarrow}Q.
\end{align*}

\begin{remk}
We note that in order for the partial $\frac{\partial}{\partial g_j}$ to be well-defined, it is important to fix a splitting $\sigma$.  Otherwise, one cannot write elements of $Q$ uniquely as sums of monomials in $\underline{g}$ with coefficients in $Q$ and different representations of an element would produce different partials.
\end{remk}

Using the map $\frac{\partial}{\partial g_j}$, one defines $\nabla$ to be the $k$-linear map
\begin{align*}
    \nabla\colon Q\underset{k[\underline{g}]}{\otimes}\Omega_{k[\underline{g}]/k}\rightarrow Q\underset{k[\underline{g}]}{\otimes}\Omega_{k[\underline{g}]/k}
\end{align*}
given by
\begin{align*}
    \nabla(q\otimes\omega)=\sum_{i=1}^s\frac{\partial}{\partial dg_i}(r)\otimes dg_i\wedge\omega+q\otimes d\omega.
\end{align*}
By our assumption that char\,$k=0$, we have that $\delta\nabla+\nabla\delta$ is invertible in nonzero degrees, so one can make the following definition; see \cite[Definition 8.8]{dyckerhoff2013}.

\begin{defn}
Let $H_{\nabla}$ be the $k$-linear map 
\begin{align*}
    H_{\nabla}&\colon Q\underset{k[\underline{g}]}{\otimes}\Omega_{k[\underline{g}]/k}\rightarrow Q\underset{k[\underline{g}]}{\otimes}\Omega_{k[\underline{g}]/k} \\
    H_{\nabla}&=(\delta\nabla+\nabla\delta)^{-1}\nabla.
\end{align*}
This map is called the \textit{de Rham contraction}.
\end{defn}

In \cite{dyckerhoff2013} Dyckerhoff and Murfet show that $H_{\nabla}$ is a homotopy on the Koszul complex.  Furthermore, they prove the following result.

\begin{thm}
\cite{dyckerhoff2013}  Under the assumptions of Section 2.3, the following is a special deformation retract datum
\begin{align*}
    \Big((Q/(\underline{g}),0)\underset{\sigma}{\overset{\pi}{\leftrightarrows}}(Q\underset{k[\underline{g}]}{\otimes}\Omega_{k[\underline{g}]/k},\delta), \thinspace H_{\nabla}\Big).
\end{align*}
\end{thm}

We modify this datum in the next section and apply the perturbation lemma to provide explicit formulas for generators of Koszul homology in the context of $\underline{g}$-weak complete intersections.

\section{Generators of Koszul Homology}
\paragraph
\indent  Let $k$ be a field of characteristic zero and $Q$ a Noetherian $k$-algebra.  Let $\underline{g}$ be a regular sequence in $Q$ such that $Q/(\underline{g})$ is a finite dimensional $k$-vector space and let $M$ be a finitely generated $\underline{g}$-weak complete intersection $Q$-module.  In this section, we study the homology of the Koszul complex on $\underline{g}$ with coefficients in $M$, which we denote as $H(\underline{g};M)$.  In particular, we provide explicit formulas for generators of each $H_i(\underline{g};M)$.  Of course, this setting includes the case where $M=Q/I$ and $I$ is a $\underline{g}$-weak complete intersection ideal. 
\newline
\indent  Now we fix some notation to be used throughout the section.   Let $(F,\partial_F)$ be a free resolution of  $M$ over $Q$ such that Im$\,\partial_F\subseteq (\underline{g})F$ as in Definition \ref{defwci}.  Let $\pi:Q\rightarrow Q/(\underline{g})$ be the usual quotient map and fix a $k$-linear splitting $\sigma\colon Q/(\underline{g})\rightarrow Q$. 
\newline
\indent In order to obtain explicit generators for the Koszul homology $H(\underline{g};M)$, we consider the isomorphisms
\begin{align}
H(\underline{g};M)\cong H(M\underset{Q}{\otimes}K(\underline{g};Q))\cong\Tor^Q(M,Q/(\underline{g}))\cong H(F\underset{Q}{\otimes}Q/(\underline{g}))\cong F\underset{Q}{\otimes}Q/(\underline{g})
\end{align} 
of $Q/(\underline{g})$-modules.  Thus, the Koszul homology we are interested in is isomorphic to the homology of the double complex $F\otimes K(\underline{g};Q)$.  We begin by giving a modification, involving a related double complex, of the special deformation retract datum of Theorem 2.12.  We apply the Perturbation Lemma to this datum to yield the desired formulas in Theorem \ref{mainthm}.  We utilize Lemmas \ref{lemmamixed}, \ref{lemmaproduct}, and \ref{lemmaconst} in the proof of this theorem and the following corollaries.

\begin{lemma}\label{lemmadr}
Let $Q$ be complete in the $(\underline{g})$-adic topology.  The following is a special deformation retract datum
\begin{align*}
    \Big((F\underset{Q}{\otimes} Q/(\underline{g}),0)\underset{1\otimes\sigma}{\overset{1\otimes\pi}{\leftrightarrows}}(\text{Tot}(F\underset{Q}{\otimes}Q\underset{k[\underline{g}]}{\otimes}\Omega_{k[\underline{g}]/k}),(0,\delta)), \thinspace 1\otimes H_{\nabla}\Big).
\end{align*}
\end{lemma}

Before the proof, we make the following remark about the maps appearing in the lemma.

\begin{remk}
Here we denote by $K$, the Koszul complex 
\begin{align*}
K=Q\underset{k[\underline{g}]}{\otimes}\Omega_{k[\underline{g}]/k}. 
\end{align*}
We note that $\sigma$ and $H_{\nabla}$ are only $k$-linear maps, so to define the maps $1\otimes\sigma$ and $1\otimes H_{\nabla}$, we first fix a basis $h_1^{\ell},...,h_{b_{\ell}}^{\ell}$ for each module $F_{\ell}$ in the resolution $F$, giving the isomorphisms $F_{\ell}\cong Q^{b_{\ell}}$.  Now we have the isomorphisms 
\begin{align*}
F_{\ell}\underset{Q}{\otimes}K_0\cong Q^{b_{\ell}}\underset{Q}{\otimes}K_0\cong K_0^{b_{\ell}}
\end{align*}
and 
\begin{align*}
F_{\ell}\underset{Q}{\otimes}Q/(\underline{g})\cong Q^{b_{\ell}}\underset{Q}{\otimes}Q/(\underline{g})\cong (Q/(\underline{g}))^{b_{\ell}}.
\end{align*}
We define $1\otimes\sigma:F_{\ell}\underset{Q}{\otimes}Q/(\underline{g})\rightarrow F_{\ell}\underset{Q}{\otimes}K_0$ by applying $\sigma$ to each of the $b_{\ell}$ summands.  We extend this to a map $F\underset{Q}{\otimes} Q/(\underline{g})\rightarrow\text{Tot}(F\underset{Q}{\otimes}Q\underset{k[\underline{g}]}{\otimes}\Omega_{k[\underline{g}]/k})$ and, abusing notation slightly, we again call this map $1\otimes\sigma$.  Similarly, given our fixed basis of $F_{\ell}$ above, we have isomorphisms
\begin{align*}
F_{\ell}\underset{Q}{\otimes}K_i\cong Q^{b_{\ell}}\underset{Q}{\otimes}K_i\cong K_i^{b_{\ell}}
\end{align*} 
and 
\begin{align*}
F_{\ell}\underset{Q}{\otimes}K_{i+1}\cong Q^{b_{\ell}}\underset{Q}{\otimes}K_{i+1}\cong K_{i+1}^{b_{\ell}},
\end{align*}
and so we define $1\otimes H_{\nabla}$ by applying $H_{\nabla}$ to each of the $b_{\ell}$ summands.  Throughout the remainder of this section, we use the notation $1\otimes\sigma$ and $1\otimes H_{\nabla}$ with the understanding that the maps are defined with respect to the fixed bases above.   Also, we note that for $i>0$, $1\otimes\pi$ sends elements of $F\underset{Q}{\otimes}K_i$ to zero. 
\newline
\indent  We will see that, in order to give a basis for the Koszul homology, it is enough to find a $k$-linear map which agrees with the $Q$-linear isomorphism (1), and apply this map to our fixed bases above.  We use the $k$-linear maps $1\otimes\sigma$ and $1\otimes H_{\nabla}$ to produce such a map.   
\end{remk}

Now we give a proof of Lemma \ref{lemmadr}.

\begin{proof}
We note that tensoring with the complex of free modules $F$ preserves the quasi-isomorphisms $\pi$ and $\sigma$.  It is also clear that $(1\otimes H_{\nabla})(1\otimes\sigma)=0$, $(1\otimes\pi )(1\otimes H_{\nabla})=0$, and $(1\otimes H_{\nabla})^2=0$ since the deformation retract datum from Theorem 2.12 is special.  So we need only check that $1\otimes H_{\nabla}$ is a homotopy between $(1\otimes\sigma)\circ (1\otimes\pi )$ and $\text{Id}_{\text{Tot}(F\underset{Q}{\otimes}Q\underset{k[\underline{g}]}{\otimes}\Omega_{k[\underline{g}]/k})}$.  It suffices to check this equality on simple tensors, so for $a_j\in F_j$ and $b_j\in K_j$ we calculate 
\begin{align*}
      (1\otimes\sigma)\circ (1\otimes\pi )(a_i\otimes b_0,...,a_0\otimes b_i)&=(1\otimes\sigma)(a_i\otimes\overline{b_0}) \\
      &=(a_i\otimes\sigma(\overline{b_0}),0,...,0).
\end{align*}
Thus, we have 
\begin{align*}
    ((1\otimes\sigma)\circ (1\otimes\pi )&-\text{Id})(a_i\otimes b_0,...,a_0\otimes b_i)=(a_i\otimes\sigma(\overline{b_0})-b_0,-a_{i-1}\otimes b_1,...,-a_0\otimes b_i) \\
    &=(a_i\otimes (\sigma\pi-\text{Id})(b_0),...,a_0\otimes (\sigma\pi-\text{Id})(b_i)) \\
    &=(a_i\otimes (\delta H_{\nabla}+H_{\nabla}\delta)(b_0),...,a_0\otimes (\delta H_{\nabla}+H_{\nabla}\delta)(b_i)) \\
    &=(\delta (1\otimes H_{\nabla})+(1\otimes H_{\nabla})\delta)(a_i\otimes b_0,...,a_0\otimes b_i)
\end{align*}
which completes the proof.
\end{proof}

In order to formulate explicit generators of Koszul homology, we need the following lemmas.

\begin{lemma}\label{lemmamixed}
For every element $q$ of $Q$,
\begin{align*}
\frac{\partial}{\partial g_j}\big(\frac{\partial}{\partial g_i}(q)\big)=\frac{\partial}{\partial g_i}\big(\frac{\partial}{\partial g_j}(q)\big).
\end{align*}
\end{lemma}

\begin{proof}
This follows directly from the definition of $\frac{\partial}{\partial g_j}$ and the fact that $\sigma$ is $k$-linear.
\end{proof}

The next lemma is a version of the product rule for $\frac{\partial}{\partial g_j}$. 

\begin{lemma}\label{lemmaproduct}
The map $\frac{\partial}{\partial g_j}$ satisfies the rule
\begin{align*}
    \frac{\partial}{\partial g_j}(qr)=\frac{\partial}{\partial g_j}(q)r+\frac{\partial}{\partial g_j}(r)q+\sum_{M,N}\Big(\frac{\partial}{\partial g_j}\big(\sigma(\bar{q_N})\sigma(\bar{r_M})\big)\Big)g^{M+N}
\end{align*}
where $g^{M+N}=g_1^{M_1+N_1}...g_s^{M_s+N_s}$ and $q=\sum_{N}\sigma(\bar{q_N})g^N$ and $r=\sum_{M}\sigma(\bar{r_M})g^M$ are elements of $Q$.
\end{lemma}

\begin{proof}
We begin by noting that $qr=\sum_{M,N}\sigma(\bar{q_N})\sigma(\bar{r_M})g^{M+N}$ and writing $\sigma(\bar{q_N})\sigma(\bar{r_M})=\sum_{P}\sigma(\bar{s_P})g^P$.  Now we compute
\begin{align*}
    \frac{\partial}{\partial g_j}(qr)&=\frac{\partial}{\partial g_j}\Bigg(\sum_{M,N,P}\sigma(\bar{s_P})g^{M+N+P}\Bigg)=\sum_{M,N,P}(M_j+N_j+P_j)\sigma(\bar{s_P})g^{M+N+P-e_j} \\
    &=\sum_{M,N,P}(M_j+N_j)\sigma(\bar{s_P})g^{M+N+P-e_j}+\sum_{M,N,P}P_j\sigma(\bar{s_P})g^{M+N+P-e_j} \\
    &=\sum_{M,N}(M_j+N_j)\Big(\sum_{P}\sigma(\bar{s_P})g^P\Big) g^{M+N-e_j}+\sum_{M,N}\Big(\sum_{P}P_j\sigma(\bar{s_P})g^{P-e_j}\Big)g^{M+N} \\
    &=\sum_{M,N}(M_j+N_j)\sigma(\bar{q_N})\sigma(\bar{r_M}) g^{M+N-e_j}+\sum_{M,N}\Bigg(\frac{\partial}{\partial g_j}\big(\sigma(\bar{q_N})\sigma(\bar{r_M})\big)\Bigg) g^{M+N} \\
\end{align*}
and we see that
\begin{align*}
    \sum_{M,N}&(M_j+N_j)\sigma(\bar{q_N})\sigma(\bar{r_M}) g^{M+N-e_j} \\
    &=\sum_{M,N}M_j\sigma(\bar{q_N})\sigma(\bar{r_M}) g^{M+N-e_j}+\sum_{M,N}N_j\sigma(\bar{q_N})\sigma(\bar{r_M}) g^{M+N-e_j} \\
    &=\Bigg(\sum_{M}M_j\sigma(\bar{r_M})g^{M-e_j}\Bigg)\Bigg(\sum_{N}\sigma(\bar{q_N})g^N\Bigg)+\Bigg(\sum_{N}N_j\sigma(\bar{q_N})g^{N-e_j}\Bigg)\Bigg(\sum_{M}\sigma(\bar{r_M})g^M\Bigg) \\
    &=\frac{\partial}{\partial g_j}(q)r+\frac{\partial}{\partial g_j}(r)q
\end{align*}
which completes the proof.
\end{proof}

\begin{remk}\label{remproduct}
We see from the lemma that $\frac{\partial}{\partial g_j}$ satisfies the usual product rule, namely
\begin{align*}
\frac{\partial}{\partial g_j}(qr)=\frac{\partial}{\partial g_j}(q)r+q\frac{\partial}{\partial g_j}(r)
\end{align*}
whenever either $q$ or $r$ can be written as $\sum_N\sigma(\bar{q_N})g^N$ where the coefficients $\sigma(\bar{q_N})$ are elements of the field $k$.  Indeed, in this case $\frac{\partial}{\partial g_j}(\sigma(\bar{q_N})\sigma(\bar{r_M}))=0$.  In particular, we have that 
\begin{align*}
\frac{\partial}{\partial g_j}(qg_k)=
\begin{cases}
\frac{\partial}{\partial g_j}(q)g_k+q & j=k \\
\frac{\partial}{\partial g_j}(q)g_k & j\neq k 
\end{cases}.
\end{align*}
However, examples which do not satisfy the usual product rule are plentiful.  For example, consider the regular sequence $g_1=x^2$, $g_2=y^3$, $g_3=z^5$ in $k[x,y,z]$.  We have that
\begin{align*}
\frac{\partial}{\partial g_1}(xy^3\cdot xz^5)=\frac{\partial}{\partial g_1}(x^2y^3z^5)=y^3z^5, 
\end{align*} 
but
\begin{align*}
\frac{\partial}{\partial g_1}(xy^3)xz^5+\frac{\partial}{\partial g_1}(xz^5)xy^3=0.
\end{align*}
Note however that 
\begin{align*}
\frac{\partial}{\partial g_1}(x\cdot x)y^3z^5=y^3z^5
\end{align*}
which illustrates the lemma.
\end{remk}

The next lemma says that the map $(\delta\nabla+\nabla\delta)^{-1}$ applied to elements of $Q\underset{k[\underline{g}]}{\otimes}\bigwedge^k\Omega^1_{k[\underline{g}]/k}$ is given in each degree by multiplication by the inverse of the sum of the internal and external degrees of that homogeneous piece.

\begin{lemma}\label{lemmaconst}
Let $\delta$ be the Koszul differential and $\nabla$ be the $k$-linear map defined in 2.3.  Then 
\begin{align*}
    (\delta\nabla+\nabla\delta)^{-1}(\sum_{N}\sigma(\bar{q_N})g^N\otimes dg_{i_1}...dg_{i_k})=\sum_{N}\Bigg(\frac{1}{|N|+k}\sigma(\bar{q_N})g^N\Bigg)\otimes dg_{i_1}...dg_{i_k}
\end{align*}
where $N=(n_1,...,n_s)$ and $|N|=n_1+...+n_s$.
\end{lemma}

\begin{proof}
Let $q=\sum_{N}\sigma(\bar{q_N})g^N$.  We begin by computing 
\begin{align*}
     &(\delta\nabla+\nabla\delta)(\sum_{N}\sigma(\bar{q_N})g^N\otimes dg_{i_1}...dg_{i_k}) \\
     &=\delta\Bigg(\sum_{j=1}^s\frac{\partial}{\partial g_j}(q)\otimes dg_jdg_{i_1}...dg_{i_k}\Bigg)+\nabla\Bigg(q\otimes\Bigg(\sum_{\ell=1}^k(-1)^{\ell+1}g_{i_{\ell}}dg_{i_1}\dots\widehat{dg_{i_{\ell}}}\dots dg_{i_k}\Bigg)\Bigg) \\      
   	&= \sum_{j=1}^s\delta\Bigg(\frac{\partial}{\partial g_j}(q)\otimes dg_jdg_{i_1}...dg_{i_k}\Bigg)+\sum_{\ell=1}^k(-1)^{\ell+1}\nabla\Big(qg_{i_{\ell}}\otimes\big(dg_{i_1}\dots\widehat{dg_{i_{\ell}}}\dots dg_{i_k}\big)\Big)\\
     &=\sum_{j=1}^s\Bigg(\frac{\partial}{\partial g_j}(q)\otimes \Bigg(g_j dg_{i_1}...dg_{i_k}+\sum_{m=1}^k(-1)^mg_{i_m}dg_{j}dg_{i_1}\dots\widehat{dg_{i_m}}\dots dg_{i_k}\Bigg) \\
     &+\sum_{\ell=1}^k(-1)^{\ell+1}\Big(\sum_{p=1}^s\frac{\partial}{\partial g_p}(qg_{i_{\ell}})\otimes\big(dg_{p}dg_{i_1}\dots\widehat{dg_{i_{\ell}}}\dots dg_{i_k}\big)\Big)\\
     &=\sum_{j=1}^s\frac{\partial}{\partial g_j}(q)g_j\otimes  dg_{i_1}...dg_{i_k}+\sum_{j=1}^s\sum_{m=1}^k(-1)^m\frac{\partial}{\partial g_j}(q)g_{i_m}\otimes dg_jdg_{i_1}\dots\widehat{dg_{i_m}}\dots dg_{i_k} \\
     &+\sum_{\ell=1}^k\sum_{p=1}^s(-1)^{\ell+1}\frac{\partial}{\partial g_p}(q)g_{i_{\ell}}\otimes dg_{p}dg_{i_1}\dots\widehat{dg_{i_{\ell}}}\dots dg_{i_k}+\sum_{\ell=1}^k(-1)^{\ell+1}q\otimes dg_{i_{\ell}}dg_{i_1}\dots\widehat{dg_{i_{\ell}}}\dots dg_{i_k}\\
\end{align*}
where the last equality follows from Remark \ref{remproduct}.  We note that the middle two sums cancel with each other and we are left with the equality 
\begin{align*}
(\delta\nabla+\nabla\delta)(\sum_{N}\sigma(\bar{q_N})g^N\otimes dg_{i_1}...dg_{i_k})
&=\sum_{j=1}^s\frac{\partial}{\partial g_j}(q)g_j\otimes  dg_{i_1}...dg_{i_k}+kq\otimes dg_{i_1}\dots dg_{i_k}.
\end{align*}
But we have that
\begin{align*}
\sum_{j=1}^s\frac{\partial}{\partial g_j}(q)g_j&=\sum_{j=1}^s\sum_{N}N_j\sigma(\bar{q_N})g^{N-e_j}g_j \\
&=\sum_{N}\sum_{j=1}^sN_j\sigma(\bar{q_N})g^N \\
&=\sum_{N}|N|\sigma(\bar{q_N})g^N,
\end{align*} 
and thus we get
\begin{align*}
(\delta\nabla+\nabla\delta)(\sum_{N}\sigma(\bar{q_N})g^N\otimes dg_{i_1}...dg_{i_k})=\sum_{N}(|N|+k)\sigma(\bar{q_N})g^N\otimes dg_{i_1}\dots dg_{i_k}.
\end{align*}
Hence, by the above equality and the fact that $\sigma$ is $k$-linear, we have 
\begin{align*}
(\delta\nabla+\nabla\delta)^{-1}&(\sum_{N}\sigma(\bar{q_N})g^N\otimes dg_{i_1}...dg_{i_k})=(\delta\nabla+\nabla\delta)^{-1}\Bigg(\sum_{N}\frac{|N|+k}{|N|+k}\sigma(\bar{q_N})g^N\otimes dg_{i_1}...dg_{i_k}\Bigg) \\
&=(\delta\nabla+\nabla\delta)^{-1}\Bigg(\sum_{N}(|N|+k)\sigma\Big(\bar{\frac{q_N}{|N|+k}}\Big)g^N\otimes dg_{i_1}...dg_{i_k}\Bigg) \\
&=\Bigg(\sum_{N}\sigma\Big(\bar{\frac{q_N}{|N|+k}}\Big)g^N\Bigg)\otimes dg_{i_1}...dg_{i_k} \\
&=\Bigg(\sum_N\frac{1}{|N|+k}\sigma(\bar{q_N})g^N\Bigg)\otimes dg_{i_1}...dg_{i_k}
\end{align*}
as desired.
\end{proof}

Based on the lemma, we establish some notation which we will use throughout the rest of the section.  

\begin{defn}
Let $q=\sum_{N}\sigma(\bar{q_N})g^N$ be an element of $Q$.  We define $\widehat{\frac{\partial}{\partial g_j}}$ by
\begin{align*}
\widehat{\frac{\partial}{\partial g_j}}\big(q\thinspace\big)=\sum_{N}\frac{N_j}{|N|}\sigma(\bar{q_N})g^{N-e_j}.
\end{align*}
\end{defn}

We are now ready to express explicit formulas for the generators of Koszul homology.

\begin{thm}\label{mainthm}
Let $Q$ be a Noetherian $k$-algebra where $k$ is a field of characteristic zero and let $\underline{g}$ be a regular sequence in $Q$ such that $Q/(\underline{g})$ is a finite dimensional $k$-vector space. If $M$ is a finitely generated $\underline{g}$-weak complete intersection module over $Q$ with free resolution $(F,\partial_F)$ as in Definition \ref{defwci}, then a $Q/(\underline{g})$-basis of $H_{\ell}(\underline{g};M)$ is given by the homology classes of the elements 
\begin{align*}
z_{\ell_{j_1}}=\sum_{k_1=1}^s...\sum_{k_{\ell}=1}^s\sum_{j_2=1}^{b_{\ell-1}}...\sum_{j_{\ell}=1}^{b_1}\overline{\widehat{\frac{\partial}{\partial g_{k_{\ell}}}}\Bigg(f_{1,j_{\ell}}^1\widehat{\frac{\partial}{\partial g_{k_{\ell-1}}}}\Bigg(f_{j_{\ell},j_{\ell-1}}^2...\widehat{\frac{\partial}{\partial g_{k_{2}}}}\Big(f_{j_{3},j_{2}}^{\ell-1}\widehat{\frac{\partial}{\partial g_{k_{1}}}}\big(f_{j_{2},j_{1}}^{\ell}\big)\Big)...\Bigg)\Bigg)}dg_{k_1}...dg_{k_{\ell}}
\end{align*} 
for $j_1=1,...,b_{\ell}$, where $h_1^i,...,h_{b_i}^i$ is a basis for $F_i$ and $\partial_F(h_{p}^i)=\sum_{m=1}^{b_{i-1}}f_{m,p}^ih_{m}^{i-1}$. 
\end{thm}

\begin{proof}
We first reduce to the case where $Q$ is complete with respect to the $(\underline{g})$-adic topology.  Suppose that we can find such a basis for $H(\underline{g};\widehat{M}^{(\underline{g})})$ over the completion $\widehat{Q/(\underline{g})}^{(\underline{g})}$.  Thus we have such a basis for
\begin{align*}
\widehat{H_{\ell}(\underline{g};M)}^{(\underline{g})}\cong H_{\ell}(\underline{g};M)\otimes\widehat{Q}^{(\underline{g})}\cong H_{\ell}(\underline{g};M\otimes\widehat{Q}^{(\underline{g})})\cong H_{\ell}(\underline{g};\widehat{M}^{(\underline{g})})
\end{align*} 
by flatness.  But 
\begin{align*}
\widehat{H_{\ell}(\underline{g};M)}^{(\underline{g})}=\underset{\underset{t}{\longleftarrow}}{\lim}\thinspace H_{\ell}(\underline{g};M)/(\underline{g})^tH_{\ell}(\underline{g};M)=H_{\ell}(\underline{g};M)
\end{align*}
where the last equality follows from the fact that $(\underline{g})\subseteq\text{ann}_Q H_{\ell}(\underline{g};M)$.  Thus we get the desired basis of $H_{\ell}(\underline{g};M)$.  
\newline
\indent Now we may assume that $Q$ is complete and thus by Lemma \ref{lemmadr} we get a special deformation retract datum
\begin{align*}
\Big((F\underset{Q}{\otimes} Q/(\underline{g}),0)\underset{1\otimes\sigma}{\overset{1\otimes\pi}{\leftrightarrows}}(\text{Tot}(F\underset{Q}{\otimes}Q\underset{k[\underline{g}]}{\otimes}\Omega_{k[\underline{g}]/k}),\delta), \thinspace 1\otimes H_{\nabla}\Big).
\end{align*} 
We note that $\partial_F\otimes 1\colon\text{Tot}(F\underset{Q}{\otimes}Q\underset{k[\underline{g}]}{\otimes}\Omega_{k[\underline{g}]/k})\rightarrow\text{Tot}(F\underset{Q}{\otimes}Q\underset{k[\underline{g}]}{\otimes}\Omega_{k[\underline{g}]/k})$, which we write as $\partial_F$ for ease of notation, is a perturbation of the special deformation retract datum above.  Indeed, the degree of $\partial_F$ is $-1$ which is the same as the degree of $\delta$ and $(\partial_F+\delta)^2=0$ since it is the differential on the total complex.  Also, since by the proof of Proposition \ref{propfpd}, $F\otimes Q/(\underline{g})$ is a finite complex, say $F_i\otimes Q/(\underline{g})=0$ for all $i>r$, we have the following equalities
\begin{align*}
\Big(1&-\partial_F(1\otimes H_{\nabla})\Big)\Big(1+\partial_F(1\otimes H_{\nabla})+\big(\partial_F(1\otimes H_{\nabla})\big)^2+...+\big(\partial_F(1\otimes H_{\nabla})\big)^r\Big) \\
&=1+\partial_F(1\otimes H_{\nabla})+\big(\partial_F(1\otimes H_{\nabla})\big)^2+...+\big(\partial_F(1\otimes H_{\nabla})\big)^r \\
&-\Big(\partial_F(1\otimes H_{\nabla})+\big(\partial_F(1\otimes H_{\nabla})\big)^2+...+\big(\partial_F(1\otimes H_{\nabla})\big)^{r+1}\Big) \\
&=1-\big(\partial_F(1\otimes H_{\nabla})\big)^{r+1}=1.
\end{align*} 
Hence, we have that
\begin{align*}
\big(1-\partial_F(1\otimes H_{\nabla})\big)^{-1}=1+\partial_F(1\otimes H_{\nabla})+\big(\partial_F(1\otimes H_{\nabla})\big)^2+...+\big(\partial_F(1\otimes H_{\nabla})\big)^r
\end{align*}
is invertible and thus $\partial_F$ is small.  By the perturbation lemma, the perturbed datum is a deformation retract datum.  In particular, we have the homotopy equivalence
\begin{align*}
    (F\underset{Q}{\otimes}Q/(\underline{g}),\tilde{0})\underset{\widetilde{1\otimes\sigma}}{\overset{\widetilde{1\otimes\pi}}{\leftrightarrows}}(\text{Tot}(F\underset{Q}{\otimes}Q\underset{k[\underline{g}]}{\otimes}\Omega_{k[\underline{g}]/k}),\delta+\partial_F).
\end{align*}
We note that $\delta+\partial_F$ is the usual differential on the total complex and that the perturbed map 
\begin{align*}
    \tilde{0}&=0+(1\otimes\pi)A(1\otimes\sigma) \\
    &=(1\otimes\pi)(1+\partial_F(1\otimes H_{\nabla})+\big(\partial_F(1\otimes H_{\nabla})\big)^2+...+\big(\partial_F(1\otimes H_{\nabla})\big)^r)\partial_F(1\otimes\sigma) \\
    &=(1\otimes\pi)\partial_F(1\otimes\sigma)
\end{align*}
where the last equality follows from the fact that $1\otimes\pi$ composed with $(\partial_F(1\otimes H_{\nabla}))^i$ for $i>0$ is zero.  Now since $\text{Im}\,\partial_F\subseteq(\underline{g})$, we get that $\pi$ composed with $\partial_F$ is zero, thus $\tilde{0}=0$, which gives the homotopy equivalence
\begin{align*}
    F\underset{Q}{\otimes}Q/(\underline{g})\overset{\widetilde{1\otimes\sigma}}{\longrightarrow}\text{Tot}(F\underset{Q}{\otimes}Q\underset{k[\underline{g}]}{\otimes}\Omega_{k[\underline{g}]/k})
\end{align*}
where 
\begin{align*}
    \widetilde{1\otimes\sigma}&=(1\otimes\sigma)+(1\otimes H_{\nabla})A(1\otimes\sigma) \\
    &=(1\otimes\sigma)+(1\otimes H_{\nabla})(1+\partial_F(1\otimes H_{\nabla})+\big(\partial_F(1\otimes H_{\nabla})\big)^2+...+\big(\partial_F(1\otimes H_{\nabla})\big)^r)\partial_F(1\otimes\sigma) \\
    &=(1\otimes\sigma)+(1\otimes H_{\nabla})\partial_F+((1\otimes H_{\nabla})\partial_F)^2+...+((1\otimes H_{\nabla})\partial_F)^r.
\end{align*}
We also note that $\widetilde{1\otimes\pi}$ is just the $Q$-linear map $1\otimes\pi$ since it sends all elements to zero except ones lying in the first row of the double complex. 
Thus, the induced map on homology 
\begin{align*}
F\underset{Q}{\otimes}Q/(\underline{g})\overset{\cong}{\rightarrow}\text{Tor}^Q(R,Q/(\underline{g}))\cong H(R\underset{Q}{\otimes}K(\underline{g};Q))\cong H(\underline{g};R)
\end{align*}
is an isomorphism whose inverse is induced by the $Q$-linear map $1\otimes\pi$, and hence agrees with the $Q$-linear isomorphism (1).  As a result, a basis for $H(\underline{g};R)$ is given by first applying $\widetilde{1\otimes\sigma}$ to the basis elements of $F{\otimes}Q/(\underline{g})$ and then taking homology classes of the results modulo $I$. 
\paragraph
\indent To this end, we compute 
\begin{align*}
((1\otimes H_{\nabla})\partial_F)^{\ell}&(h_{j_1}^{\ell}\otimes 1\otimes 1)=((1\otimes H_{\nabla})\partial_F)^{\ell-1}(1\otimes H_{\nabla})\Big(\sum_{j_2=1}^{b_{\ell-1}}f_{j_2,j_1}^{\ell}h_{j_2}^{\ell-1}\otimes 1\otimes 1\Big) \\
&=\sum_{j_2=1}^{b_{\ell-1}}((1\otimes H_{\nabla})\partial_F)^{\ell-1}(1\otimes H_{\nabla})(h_{j_2}^{\ell-1}\otimes f_{j_2,j_1}^{\ell}\otimes 1) \\
&=\sum_{j_2=1}^{b_{\ell-1}}((1\otimes H_{\nabla})\partial_F)^{\ell-1}(h_{j_2}^{\ell-1}\otimes \sum_{k_1=1}^s\widehat{\frac{\partial}{\partial g_{k_1}}}(f_{j_2,j_1}^{\ell})\otimes dg_{k_1}) \\
&=\sum_{k_1=1}^s\sum_{j_2=1}^{b_{\ell-1}}((1\otimes H_{\nabla})\partial_F)^{\ell-1}(h_{j_2}^{\ell-1}\otimes \widehat{\frac{\partial}{\partial g_{k_1}}}(f_{j_2,j_1}^{\ell})\otimes dg_{k_1})
\end{align*}
Applying this procedure $\ell-1$ more times, we get that
\begin{align*}
&((1\otimes H_{\nabla})\partial_F)^{\ell}(h_{j_1}^{\ell}\otimes 1\otimes 1)= \\
&-\sum_{k_1=1}^s...\sum_{k_{\ell}=1}^s\sum_{j_2=1}^{b_{\ell-1}}...\sum_{j_{\ell}=1}^{b_1}\Bigg(1\otimes\widehat{\frac{\partial}{\partial g_{k_{\ell}}}}\Bigg(f_{1,j_{\ell}}^1\widehat{\frac{\partial}{\partial g_{k_{\ell-1}}}}\Bigg(f_{j_{\ell},j_{\ell-1}}^2...\widehat{\frac{\partial}{\partial g_{k_{1}}}}\big(f_{j_{2},j_{1}}^{\ell}\big)...\Bigg)\Bigg)\otimes dg_{k_1}...dg_{k_{\ell}}\Bigg)
\end{align*}
which completes the proof.
\end{proof}

We now give useful versions of these formulas for some special cases of interest.  For this we establish some notation.  We denote by $\frac{\partial(f_1,...,f_i)}{\partial(g_{k_1},...,g_{k_i})}$ the determinant of the $i\times i$ matrix $(\frac{\partial}{\partial g_{k_j}}(f_{\ell}))_{j,\ell}$.  

\begin{defn}
We call $f_{i,j}^k$ a \textit{homogeneous polynomial in} $\underline{g}$ \textit{with coefficients in} $Q$ if there is an integer $n$ such that
\begin{align*}
f_{i,j}^k=\sum_{N}\sigma(\overline{{f_{i,j}^k}_N)}g^N
\end{align*}
for $N=(n_1,...,n_s)\in\mathbb{N}^s$ satisfying $n_1+...+n_s=n$.  In this case, we call $n$ the $\underline{g}$\textit{-degree} of $f_{i,j}^k$ and we denote it by $d_{i,j}^k$. 
\end{defn}

\begin{cor}\label{cor1}
Let $Q$ be a Noetherian $k$-algebra with $k$ a field of characteristic zero and let $\underline{g}$ be a regular sequence in $Q$ such that $Q/(\underline{g})$ is a finite-dimensional $k$-vector space. If $M$ is a $\underline{g}$-weak complete intersection $Q$-module with free resolution $(F,\partial_F)$ such that the entries in the matrices $\partial_F$ are homogeneous polynomials in $\underline{g}$ with coefficients in $Q$ , then a $Q/(\underline{g})$-basis of $H_{\ell}(\underline{g};M)$ is given by the homology classes of the elements 
\begin{align*}
&z_{\ell_{j_1}}=\sum_{j_2=1}^{b_{\ell-1}}...\sum_{j_{\ell}=1}^{b_1}D_{j_1,...,j_{\ell}}\Bigg(\sum_{1\leq k_1<...<k_{\ell}\leq s}\overline{\frac{\partial(f_{1,j_{\ell}}^1,f_{j_{\ell},j_{\ell-1}}^2,...,f_{j_2,j_1}^{\ell})}{\partial(g_{k_1},...,g_{k_{\ell}})}} \\
&+\sum_{k_1=1}^s...\sum_{k_{\ell}=1}^s\sum_{m=2}^{\ell}\prod_{n=m+1}^{\ell}\overline{\frac{\partial}{\partial g_{k_n}}(f_{j_{n+1},j_n}^{\ell-n+1})\sum_{M,N}\frac{\partial}{\partial g_{k_m}}\Big(\sigma\big(\overline{{f_{j_{m+1},j_m}^{\ell-m+1}}_M}\big)\sigma\big(\overline{y_{(m-1)_{j_1}}}_N\big)\Big)g^{M+N}}\Bigg)dg_{k_1}...dg_{k_{\ell}}
\end{align*}
for $j_1=1,...,b_{\ell}$, where
\begin{align*}
z_{(m-1)_{j_1}}=\sum_{k_1=1}^s...\sum_{k_{m-1}=1}^s\sum_{j_2=1}^{b_{m-2}}...\sum_{j_{m-1}=1}^{b_1}D_{j_1,...,j_{m-1}}y_{(m-1)_{j_1}}dg_{k_1}...dg_{k_{m-1}}
\end{align*}
and where
\begin{align*}
D_{j_1,...,j_{\ell}}=\frac{1}{d_{j_2,j_1}^{\ell}}\cdot\frac{1}{d_{j_3,j_2}^{\ell-1}+d_{j_2,j_1}^{\ell}-1}\cdot...\cdot\frac{1}{d_{1,j_{\ell}}^{1}+d_{j_{\ell},j_{\ell-1}}^{2}+...+d_{j_2,j_1}^{\ell}-\ell+1}.
\end{align*}
\end{cor}

\begin{proof}
We first note that for a $\underline{g}$-homogeneous polynomial, $f_{i,j}^k$ we have that
\begin{align*}
    \widehat{\frac{\partial}{\partial g_j}}(f_{i,j}^k)=\frac{1}{d_{i,j}^k}\frac{\partial}{\partial g_j}(f_{i,j}^k).
\end{align*}
Thus, since $\frac{\partial}{\partial g_j}$ reduces the $\underline{g}$-degree of its argument by one whenever it is nonzero, the formulas from Theorem \ref{mainthm} become $z_{\ell_{j_1}}=$
\begin{align}
    \sum_{k_1=1}^s...\sum_{k_{\ell}=1}^s\sum_{j_2=1}^{b_{\ell-1}}...\sum_{j_{\ell}=1}^{b_1}D_{j_1,...,j_{\ell}}\frac{\partial}{\partial g_{k_{\ell}}}\Bigg(f_{1,j_{\ell}}^1\frac{\partial}{\partial g_{k_{\ell-1}}}\Bigg(f_{j_{\ell},j_{\ell-1}}^2...\frac{\partial}{\partial g_{k_{1}}}\big(f_{j_{2},j_{1}}^{\ell}\big)...\Bigg)\Bigg)dg_{k_1}...dg_{k_{\ell}}.  
\end{align}
where we omit the bar for ease of exposition.  It is now clear that we have the desired formula for $z_{1_{j_1}}$.  To obtain the desired formula for $z_{\ell_{j_1}}$, we apply the product rule from Lemma \ref{lemmaproduct} to (2), to get 
\begin{align*}
    \sum_{k_1=1}^s&...\sum_{k_{\ell}=1}^s\sum_{j_2=1}^{b_{\ell-1}}...\sum_{j_{\ell}=1}^{b_1}D_{j_1,...,j_{\ell}}\Bigg(\frac{\partial}{\partial g_{k_{\ell}}}\big(f_{1,j_{\ell}}^1\big)\frac{\partial}{\partial g_{k_{\ell-1}}}\Bigg(f_{j_{\ell},j_{\ell-1}}^2...\frac{\partial}{\partial g_{k_{1}}}\big(f_{j_{2},j_{1}}^{\ell}\big)...\Bigg) \\
    &+f_{1,j_{\ell}}^1\frac{\partial}{\partial g_{k_{\ell}}}\Bigg(\frac{\partial}{\partial g_{k_{\ell-1}}}\Bigg(f_{j_{\ell},j_{\ell-1}}^2...\frac{\partial}{\partial g_{k_{1}}}\big(f_{j_{2},j_{1}}^{\ell}\big)...\Bigg)\Bigg) \\
    &+\sum_{M,N}\frac{\partial}{\partial g_{k_{\ell}}}\Bigg(\sigma\big(\overline{f_{1,j_{\ell}}^1}_M\big)\sigma\Bigg(\overline{\frac{\partial}{\partial g_{k_{\ell-1}}}\Bigg(f_{j_{\ell},j_{\ell-1}}^2...\frac{\partial}{\partial g_{k_{1}}}\big(f_{j_{2},j_{1}}^{\ell}\big)...\Bigg)}_N\Bigg)\Bigg)g^{M+N}\Bigg)dg_{k_1}...dg_{k_{\ell}}.
\end{align*}
But by Lemma \ref{lemmamixed}, we have that 
\begin{align*}
    \sum_{k_1=1}^s&...\sum_{k_{\ell}=1}^s\frac{\partial}{\partial g_{k_{\ell}}}\Bigg(\frac{\partial}{\partial g_{k_{\ell-1}}}\Bigg(f_{j_{\ell},j_{\ell-1}}^2...\frac{\partial}{\partial g_{k_{1}}}\big(f_{j_{2},j_{1}}^{\ell}\big)\Big)...\Bigg)\Bigg)dg_{k_1}...dg_{k_{\ell}}=0,
\end{align*}
which gives the following formula for $z_{\ell_{j_1}}$
\begin{align*}
     &\sum_{k_1=1}^s...\sum_{k_{\ell}=1}^s\sum_{j_2=1}^{b_{\ell-1}}...\sum_{j_{\ell}=1}^{b_1}D_{j_1,...,j_{\ell}}\Bigg(\frac{\partial}{\partial g_{k_{\ell}}}\big(f_{1,j_{\ell}}^1\big)\frac{\partial}{\partial g_{k_{\ell-1}}}\Bigg(f_{j_{\ell},j_{\ell-1}}^2...\frac{\partial}{\partial g_{k_{1}}}\big(f_{j_{2},j_{1}}^{\ell}\big)...\Bigg) \\
     &+\sum_{M,N}\frac{\partial}{\partial g_{k_{\ell}}}\Bigg(\sigma\big(\overline{f_{1,j_{\ell}}^1}_M\big)\sigma\Bigg(\overline{\frac{\partial}{\partial g_{k_{\ell-1}}}\Bigg(f_{j_{\ell},j_{\ell-1}}^2...\frac{\partial}{\partial g_{k_{1}}}\big(f_{j_{2},j_{1}}^{\ell}\big)...\Bigg)}_N\Bigg)\Bigg)g^{M+N}\Bigg)dg_{k_1}...dg_{k_{\ell}}.
\end{align*}
Now we apply the product rule repeatedly, simplifying at each step as above to obtain
\begin{align*}
    &\sum_{k_1=1}^s...\sum_{k_{\ell}=1}^s\sum_{j_2=1}^{b_{\ell-1}}...\sum_{j_{\ell}=1}^{b_1}D_{j_1,...,j_{\ell}}\Bigg(\frac{\partial}{\partial g_{k_{\ell}}}\big(f_{1,j_{\ell}}^1\big)\frac{\partial}{\partial g_{k_{\ell-1}}}\big(f_{j_{\ell},j_{\ell-1}}^1\big)...\frac{\partial}{\partial g_{k_1}}\big(f_{j_{2},j_{1}}^{\ell}\big) \\
    &+\sum_{m=2}^{\ell}\prod_{n=m+1}^{\ell}\overline{\frac{\partial}{\partial g_{k_n}}(f_{j_{n+1},j_n}^{\ell-n+1})\sum_{M,N}\frac{\partial}{\partial g_{k_m}}\Big(\sigma\big(\overline{{f_{j_{m+1},j_m}^{\ell-m+1}}_M}\big)\sigma\big(\overline{y_{(m-1)_{j_1}}}_N\big)\Big)g^{M+N}}\Bigg)dg_{k_1}...dg_{k_{\ell}}
\end{align*}
where 
\begin{align*}
    y_{(m-1)_{j_1}}=\frac{\partial}{\partial g_{k_{m-1}}}\Bigg(f_{j_{m},j_{m-1}}^{\ell-m+2}...\frac{\partial}{\partial g_{k_{1}}}\big(f_{j_{2},j_{1}}^{\ell}\big)...\Bigg).
\end{align*}
We observe that the first part of the above formula can be written as a determinant, thus giving the desired formulas.
\end{proof}

We now give formulas in the case that the entries in the matrices given by $\partial_F$ are homogeneous polynomials in $\underline{g}$ with coefficients in $k$ rather than $Q$.

\begin{cor}\label{cor2}
Let $Q$ be a Noetherian $k$-algebra with $k$ a field of characteristic zero and let $\underline{g}=g_1,...,g_s$ be a regular sequence in $Q$. If $M$ is a $\underline{g}$-weak complete intersection module over $Q$ with free resolution $(F,\partial_F)$ such that the entries in the matrices $\partial_F$ are homogeneous polynomials in $\underline{g}$ with coefficients in $k$, then a $Q/(\underline{g})$-basis of $H_{\ell}(\underline{g};M)$ is given by the homology classes of the elements
\begin{align*}
    z_{\ell_{j_1}}&=\sum_{j_2=1}^{b_{\ell-1}}...\sum_{j_{\ell}=1}^{b_1}D_{j_1,...,j_{\ell}}\sum_{1\leq k_1<...<k_{\ell}\leq s}\overline{\frac{\partial(f_{1,j_{\ell}}^1,f_{j_{\ell},j_{\ell-1}}^2,...,f_{j_2,j_1}^{\ell})}{\partial(g_{k_1},...,g_{k_{\ell}})}}dg_{k_1}...dg_{k_{\ell}}
\end{align*}
for $j_1=1,...,b_{\ell}$.
\end{cor}

\begin{proof}
This follows immediately from Corollary \ref{cor1} and the fact that $\frac{\partial}{\partial g_j}$ satisfies the usual product rule in the case that the coefficients involved in the product are elements of $k$ as discussed in Remark \ref{remproduct}.
\end{proof}

\begin{remk}
Taking $\underline{g}=x_1,...,x_n$ to be the minimal generators of the maximal ideal, the corollary recovers the formulas given by Herzog in \cite{MR1310371}.  
\end{remk}

We finish this section by giving an example which illustrates our results.

\begin{ex}
Let $Q=k[x,y,z]$, $\underline{g}$ be the regular sequence $g_1=x^2+yz$, $g_2=y^3$, $g_3=z^5$, and let $I=(x^2y^4+y^5z+xz^{10}, y^6, x^4y^2+x^2y^3z)$.  Let us compute the generators of $H_{\ell}(\underline{g};R)$, where $R=Q/I$.  We begin by fixing a $k$-linear splitting of $\pi:Q\rightarrow Q/(\underline{g})$.  One can find the following basis for $Q/(\underline{g})$ as a $k$-vector space either by hand or using Macaulay2
\begin{align*}
    &\bar{1}, \thinspace\overline{x}, \thinspace\overline{xy}, \thinspace\overline{xy^2}, \thinspace\overline{xy^2z}, \thinspace\overline{xy^2z^2}, \thinspace\overline{xy^2z^3}, \thinspace\overline{xy^2z^4}, \thinspace\overline{xyz}, \thinspace\overline{xyz^2}, \thinspace\overline{xyz^3}, \thinspace\overline{xyz^4}, \thinspace\overline{xz}, \thinspace\overline{xz^2}, \\
    &\thinspace\overline{xz^3}, \thinspace\overline{xz^4}, \thinspace\overline{y}, \thinspace\overline{y^2}, \thinspace\overline{y^2z}, \thinspace\overline{y^2z^2}, \thinspace\overline{y^2z^3}, \thinspace\overline{y^2z^4}, \thinspace\overline{yz}, \thinspace\overline{yz^2}, \thinspace\overline{yz^3}, \thinspace\overline{yz^4}, \thinspace\overline{z}, \thinspace\overline{z^2}, \thinspace\overline{z^3}, \thinspace\overline{z^4}.
\end{align*}
We choose the splitting $\sigma(\bar{a})=a$ for every basis element $\bar{a}$.  According to Macaulay2, a free resolution of $R$ over $Q$ is given by
\begin{align*}
    0\rightarrow Q\overset{\partial_3}{\longrightarrow} Q^3\overset{\partial_2}{\longrightarrow} Q^3\overset{\partial_1}{\longrightarrow} Q\rightarrow R\rightarrow 0
\end{align*}
where the differentials are given by the following matrices:
\begin{align*}
\partial_3=
\left[
\begin{array}{c}
z^{10}+xy^4 \\
-y^4 \\
x^3+xyz
\end{array}
\right],\thinspace\thinspace\thinspace
&\partial_2=
\left[
\begin{array}{c c c}
-y^4 & -z^{10}-xy^4 & 0 \\
x^4+x^2yz & -x^3yz-xy^2z^2 & -xz^{10}-x^2y^4-y^5z \\
0 & x^3y^2+xy^3z & y^6
\end{array}
\right], \\
\vspace{0.025cm} \\
\partial_1=&
\left[
\begin{array}{c c c}
x^4y^2+x^2y^3z & y^6 & x^2y^4+y^5z+xz^{10} \\
\end{array}
\right].
\end{align*}
We now write the entries in the differentials as in Lemma \ref{lemmalipman}.  Note for example that the first entry in $\partial_1$ can be written as follows:
\begin{align*}
    x^4y^2+x^2y^3z=x^2y^2g_1
\end{align*}
but note that $x^2y^2$ is not in the image of our splitting map $\sigma$, so we write
\begin{align*}
    x^4y^2+x^2y^3z&=x^2y^2g_1 \\
    &=(g_1-yz)y^2g_1 \\
    &=y^2g_1^2-y^3zg_1 \\
    &=y^2g_1^2-zg_1g_2
\end{align*}
and we can see that the coefficients are now in the image of $\sigma$.  Using a similar procedure on the other entries, we obtain the following matrices
\begin{align*}
    \partial_3=
\left[
\begin{array}{c}
xyg_2+g_3^2 \\
-yg_2 \\
xg_1
\end{array}
\right],\thinspace\thinspace\thinspace
&\partial_2=
\left[
\begin{array}{c c c}
-yg_2 & -xyg_2-g_3^2 & 0 \\
g_1^2-yzg_1 & -xyzg_1 & -yg_1g_2-xg_3^2 \\
0 & xy^2g_1 & g_2^2
\end{array}
\right], \\
\vspace{0.025cm} \\
\partial_1=&
\left[
\begin{array}{c c c}
y^2g_1^2-zg_1g_2 & g_2^2 & yg_1g_2+xg_3^2 \\
\end{array}
\right].
\end{align*}
Applying Theorem \ref{mainthm}, we get the following set of elements whose homology classes generate $H_1(\underline{g};R)$,
\begin{align*}
    \widetilde{h_1^1}&=\overline{y^2g_1-zg_2}dg_1-\frac{1}{2}\overline{zg_1}dg_2=\overline{x^2y}dg_1-\frac{1}{2}\overline{x^2z+yz^2}dg_2 \\
    \widetilde{h_2^1}&=\overline{g_2}dg_2=\overline{y^3}dg_2 \\
    \widetilde{h_3^1}&=\frac{1}{2}\overline{yg_2}dg_1+\frac{1}{2}\overline{yg_1}dg_2+\overline{xg_3}dg_3=\frac{1}{2}\overline{y^4}dg_1-\frac{1}{2}\overline{x^2y+y^2z}dg_2+\overline{xz^5}dg_3
\end{align*}
where $\widetilde{h_i^j}$ is the generator which corresponds to the the basis element $h_i^j$ of $F_j$.  We get the following generators for $H_2(\underline{g};R)$
\begin{align*}
    \widetilde{h_1^2}&=-\frac{1}{2}\overline{y^4z}dg_1dg_2 \\
    \widetilde{h_2^2}&=\big(-\frac{1}{2}\overline{xy^4z}-\frac{1}{3}\overline{z^{11}}\big)dg_1dg_2-\frac{1}{3}\overline{y^3z^6}dg_1dg_3+\frac{1}{3}\overline{x^2z^6+yz^7}dg_2dg_3 \\
    \widetilde{h_3^2}&=-\frac{1}{3}\overline{y^7}dg_1dg_2
\end{align*}
and for $H_3(\underline{g};R)$
\begin{align*}
    \widetilde{h_1^3}&=\big(\frac{1}{6}\overline{x^2y^3z^5}+\frac{2}{3}\overline{y^4z^6}+\frac{1}{9}\overline{x^2z^5}-\frac{1}{18}\overline{yz^6}\big)dg_1dg_2dg_3.
\end{align*}
\end{ex}

\vspace{0.25cm}

\section{Applications to Weak Complete Intersection Ideals}
\paragraph
\indent Let $Q$ be a Noetherian $k$-algebra with $k$ a field of characteristic zero and let $\underline{g}=g_1,...,g_s$ be a regular sequence in $Q$.  In this section, we use the explicit formulas for generators of Koszul homology given in the previous section to study the ideal $(\underline{g})$ of the quotient $R=Q/I$, where $I$ is a $\underline{g}$-weak complete intersection ideal of $Q$.  In particular, we are interested in answering the following question.  

\begin{quest}
When is the ideal $(\underline{g})$ of $R$ a weak complete intersection ideal?
\end{quest}

In Proposition \ref{propapp}, we provide a general condition under which $(\underline{g})$ is a weak complete intersection ideal in $R$, which expands the class of known examples of weak complete intersection ideals.  However, there are examples of ideals $(\underline{g})$ which are not weak complete intersection ideals in quotients by $\underline{g}$-weak complete intersection ideals.  We discuss one example in the following remark.  

\begin{remk}
The ideal $(\underline{g})$ of $R=Q/I$ in Example \ref{expowersvars} is not a weak complete intersection ideal.  One can verify this by looking at the beginning of the minimal free resolution of $R/(\underline{g})$ over $R$ on Macaulay2.  There are entries in the differentials which are not elements of $(\underline{g})$, for example $y^2$ is one such entry. 
\end{remk}

In order to study this question further, we will need the following definition; see for example \cite[Definition 2.7]{RAHMATI2018129} or \cite[Remark 5.2.1]{MR2641236}.

\begin{defn}
Let $R$ be a local ring and assume $H_{\ell}(\underline{g};R)$ is a free $R/(g)$-module for all $\ell$.  We say that $K(\underline{g};R)$ admits a \textit{trivial Massey operation} if for some basis $\mathcal{B}=\{h_{\lambda}\}_{\lambda\in\Lambda}$ of $H_{\geq 1}(\underline{g};R)$, there is a function
\begin{align*}
\mu:\bigsqcup_{n=1}^{\infty}\mathcal{B}^n\rightarrow K(\underline{g};R) 
\end{align*}
such that $\mu(h_{\lambda})=z_{\lambda}$ is a cycle with cls$(z_{\lambda})=h_{\lambda}$ and
\begin{align}
\partial^K\mu(h_{\lambda_1},...,h_{\lambda_p})=\sum_{j=1}^{p-1}\overline{\mu(h_{\lambda_1},...,h_{\lambda_j})}\mu(h_{\lambda_{j+1}},...,h_{\lambda_p})
\end{align}
where $\overline{a}=(-1)^{|a|+1}a$.
\end{defn}

Now we establish some notation.

\begin{defn}
Let $I$ be an ideal in $Q$.  We define $\frac{\partial}{\partial g}(I)$ to be the ideal generated by the elements $\{\frac{\partial}{\partial g_j}(f)|f\in I, j=1,...,s\}$. 
\end{defn}

The following result sheds some light on the question stated above.  See \cite[Theorem 3.5]{MR3633302} and \cite[Theorem 1.1]{HERZOG201389} for similar results regarding Golod rings and modules.

\begin{prop}\label{propapp}
Let $Q$ be a local Noetherian $k$-algebra with $k$ a field of characteristic zero and let $\underline{g}$ be a regular sequence in $Q$ such that $Q/(\underline{g})$ is a finite-dimensional $k$-vector space.  If $I$ is a $\underline{g}$-weak complete intersection ideal of $Q$ and $\frac{\partial}{\partial g}(I)^2\subseteq I$, then $(\underline{g})$ is a weak complete intersection ideal in $R=Q/I$. 
\end{prop}

\begin{proof}
Since $(\underline{g})/(\underline{g})^2$ is a free $Q/(\underline{g})$-module and $H_{\ell}(\underline{g};R)$ is a free $Q/(\underline{g})$-module for every $i$, it suffices to show that $K(\underline{g};R)$ admits a trivial Massey operation by \cite[Theorem 2.9]{RAHMATI2018129}.\newline
\indent
We take $\mathcal{B}$ to be the basis of $H(\underline{g};R)$ given in Theorem \ref{mainthm} and lift these basis elements to cycles $z_{\lambda}$ using the formulas in the theorem.  We define $\mu(h_{\lambda_1},...,h_{\lambda_p})=0$ for all $p\geq 2$ and $h_{\lambda_i}\in\mathcal{B}$.  By Theorem \ref{mainthm}, we see that every $z_{\lambda}$ has coefficients in $\frac{\partial}{\partial g}(I)$, since the elements $f_{1,j_{\ell}}^1$ are the entries in the first differential in the minimal free resolution of $R$, and hence are elements of $I$.  Thus, the coefficients of $\mu(h_{\lambda_i})\mu(h_{\lambda_j})$ are elements of $\frac{\partial}{\partial g}(I)^2\subseteq I$ for all $i$ and $j$, so the products are zero in $K(\underline{g};R)$.  It is now easy to see that our definition of $\mu$ satisfies
\begin{align*}
\partial^K\mu(h_{\lambda_1},...,h_{\lambda_p})=\sum_{j=1}^{p-1}\overline{\mu(h_{\lambda_1},...,h_{\lambda_j})}\mu(h_{\lambda_{j+1}},...,h_{\lambda_p}).
\end{align*}
Thus, $\mu$ is a trivial Massey operation on $K(\underline{g};R)$, as desired. 
\end{proof}

The following example illustrates Proposition \ref{propapp} and shows that it produces new examples of weak complete intersection ideals.

\begin{ex}
Let $Q=k[x,y]$ with char\,$k=0$ and let $\underline{g}$ be the regular sequence $g_1=x^2+y^2$, $g_2=y^3$.  We consider the ideal $I=(g_1^2g_2, g_1^4, g_2^3)\subseteq Q$.  A free resolution of $Q/I$ is 
\begin{align*}
0\rightarrow Q^2\overset{\partial_2}{\longrightarrow} Q^3\overset{\partial_1}{\longrightarrow} Q\rightarrow R\rightarrow 0
\end{align*}
with differentials
\begin{align*}
&\partial_2=
\left[
\begin{array}{c c}
-g_1^2 & -g_2^2 \\
g_2 & 0 \\
0 & g_1^2
\end{array}
\right],\thinspace\thinspace\thinspace
\partial_1=
\left[
\begin{array}{c c c}
g_1^2g_2 & g_1^4 & g_2^3 \\
\end{array}
\right],
\end{align*}
so $I$ is a $\underline{g}$-weak complete intersection ideal.  The ideal $\frac{\partial}{\partial g}(I)$ is generated by the elements $g_1g_2, g_1^2, g_2^2$.  Indeed, elements of $I$ are of the form $f=ag_1^2g_2+bg_1^4+cg_2^3$ for $a,b,c\in Q$, and 
\begin{align*}
\frac{\partial}{\partial g_1}(f)&=2ag_1g_2+4bg_1^3+\frac{\partial}{\partial g_1}(c)g_2^3 \\
\frac{\partial}{\partial g_2}(f)&=ag_1^2+\frac{\partial}{\partial g_2}(b)g_1^4+3cg_2^2
\end{align*}
which are both elements of the ideal $(g_1g_2, g_1^2, g_2^2)$.  It is easy to check that $\frac{\partial}{\partial g}(I)^2\subseteq I$.  Thus, by Proposition \ref{propapp}, $(\underline{g})$ is a weak complete intersection ideal in $R=Q/I$.  One can also verify this fact by looking at a free resolution for $R/(\underline{g})$ over $R$.  
\end{ex}

The converse of Proposition \ref{propapp} is not true in general, as shown by the following example.

\begin{ex}\label{exconverse}
Let $Q=k[x,y,z]$ with char\,$k$=0 and let $\underline{g}$ be the regular sequence $g_1=x^2$, $g_2=y^3$, $g_3=z^5$.  The ideal $I=(x^2y^8,y^8z^9,x^3z^{14}+x^5y^5)\subseteq Q$ is a $\underline{g}$-weak complete intersection ideal; a free resolution of $R=Q/I$ being
\begin{align*}
0\rightarrow Q^2\overset{\partial_2}{\longrightarrow} Q^3\overset{\partial_1}{\longrightarrow} Q\rightarrow R\rightarrow 0
\end{align*}
with differentials
\begin{align*}
&\partial_2=
\left[
\begin{array}{c c}
-z^9 & -x^3y^5 \\
x^2 & -x^3z^5 \\
0 & y^8
\end{array}
\right],\thinspace\thinspace\thinspace
\partial_1=
\left[
\begin{array}{c c c}
x^2y^8 & y^8z^9 & x^3z^{14}+x^5y^5 \\
\end{array}
\right].
\end{align*}
By Theorem \ref{mainthm}, the homology classes of the elements $\{\tilde{h_1^1},\tilde{h_2^1},\tilde{h_3^1},\tilde{h_1^2},\tilde{h_2^2}\}$, where
\begin{align*}
\tilde{h_1^1}&=\frac{1}{3}\overline{y^8}dg_1+\frac{2}{3}\overline{x^2y^5}dg_2 \\
\tilde{h_2^1}&=\frac{2}{3}\overline{y^5z^9}dg_2+\frac{1}{3}\overline{y^8z^4}dg_3 \\
\tilde{h_3^1}&=\frac{1}{3}\overline{(xz^{14}+2x^3y^5)}dg_1+\frac{1}{3}\overline{x^5y^2}dg_2+\frac{2}{3}\overline{x^3z^9}dg_3 \\
\tilde{h_1^2}&=\frac{2}{3}\overline{y^5z^9}dg_1dg_2-\frac{2}{3}\overline{x^2y^5z^4}dg_2dg_3 \\
\tilde{h_2^2}&=\frac{1}{6}\overline{x^3y^{10}}dg_1dg_2-\frac{1}{6}\overline{x^2y^8z^4}dg_1dg_3+\frac{1}{3}\overline{x^3y^5z^9}dg_2dg_3 
\end{align*}
is a basis for $H_{\geq 1}(\underline{g};R)$.  To obtain a trivial Massey operation on $K(\underline{g};R)$, we first multiply the cycles
\begin{align*}
\tilde{h_1^1}\cdot\tilde{h_2^1}&=\frac{1}{9}\overline{y^{16}z^4}dg_1dg_3 \\
\tilde{h_1^1}\cdot\tilde{h_3^1}&=\frac{4}{9}\overline{x^5y^{5}z^9}dg_2dg_3 \\
\tilde{h_2^1}\cdot\tilde{h_3^1}&=\frac{2}{9}\overline{xy^{5}z^{23}}dg_1dg_2
\end{align*} 
and all multiplications involving $\tilde{h_1^2}$ and $\tilde{h_2^2}$ are zero.  So we define $\mu$ as follows
\begin{align*}
\mu([\tilde{h_1^1}],[\tilde{h_2^1}])=-\frac{1}{9}y^{13}z^4dg_1dg_2d_3 \\
\mu([\tilde{h_1^1}],[\tilde{h_3^1}])=\frac{4}{9}x^3y^{5}z^9dg_1dg_2d_3 \\
\mu([\tilde{h_2^1}],[\tilde{h_2^1}])=\frac{2}{9}xy^{5}z^{18}dg_1dg_2d_3
\end{align*}
where $[\cdot]$ denotes the homology class, and otherwise we define $\mu$ to be zero.  It is straightforward to check that $\mu$ satisfies (3), thus it is a trivial Massey operation.  Therefore, by \cite[Theorem 2.9]{RAHMATI2018129}, $(\underline{g})$ is a weak complete intersection ideal in $R$.  However, $y^{8}\in\frac{\partial}{\partial g}(I)$, thus $y^{16}$ is in $\frac{\partial}{\partial g}(I)^2$, but not in $I$.  This shows that $\frac{\partial}{\partial g}(I)^2\subseteq I$ is not a necessary condition for $(\underline{g})$ to be a weak complete intersection ideal in $R$.   
\end{ex}

\section*{\centering\normalsize{\normalfont{\uppercase{Acknowledgments}}}}
The author would like to thank her advisor, Claudia Miller, for all of her guidance, help, and support on this project.

\bibliographystyle{siam}
\bibliography{generatorsarxivv1}

\begin{thebibliography}{10}

\bibitem{MR2641236}
{\sc L.~L. Avramov}, {\em Infinite free resolutions [mr1648664]}, in Six
  lectures on commutative algebra, Mod. Birkh\"{a}user Class., Birkh\"{a}user
  Verlag, Basel, 2010, pp.~1--118.

\bibitem{MR3717974}
{\sc A.~Corso, S.~Goto, C.~Huneke, C.~Polini, and B.~Ulrich}, {\em Iterated
  socles and integral dependence in regular rings}, Trans. Amer. Math. Soc.,
  370 (2018), pp.~53--72.

\bibitem{2004math......3266C}
{\sc M.~{Crainic}}, {\em {On the perturbation lemma, and deformations}}, arXiv
  Mathematics e-prints,  (2004), p.~math/0403266.

\bibitem{dyckerhoff2013}
{\sc T.~Dyckerhoff and D.~Murfet}, {\em Pushing forward matrix factorizations},
  Duke Math. J., 162 (2013), pp.~1249--1311.

\bibitem{MR3633302}
{\sc A.~Gupta}, {\em Ascent and descent of the {G}olod property along algebra
  retracts}, J. Algebra, 480 (2017), pp.~124--143.

\bibitem{MR1310371}
{\sc J.~Herzog}, {\em Canonical {K}oszul cycles}, in International {S}eminar on
  {A}lgebra and its {A}pplications ({S}panish) ({M}\'{e}xico {C}ity, 1991),
  vol.~6 of Aportaciones Mat. Notas Investigaci\'{o}n, Soc. Mat. Mexicana,
  M\'{e}xico, 1992, pp.~33--41.

\bibitem{HERZOG201389}
{\sc J.~Herzog and C.~Huneke}, {\em Ordinary and symbolic powers are golod},
  Advances in Mathematics, 246 (2013), pp.~89 -- 99.

\bibitem{Herzog2018}
{\sc J.~Herzog and R.~A. Maleki}, {\em Koszul cycles and golod rings},
  manuscripta mathematica, 157 (2018), pp.~483--495.

\bibitem{MR868864}
{\sc J.~Lipman}, {\em Residues and traces of differential forms via
  {H}ochschild homology}, vol.~61 of Contemporary Mathematics, American
  Mathematical Society, Providence, RI, 1987.

\bibitem{MR1011461}
{\sc H.~Matsumura}, {\em Commutative ring theory}, vol.~8 of Cambridge Studies
  in Advanced Mathematics, Cambridge University Press, Cambridge, second~ed.,
  1989.
\newblock Translated from the Japanese by M. Reid.

\bibitem{RAHMATI2018129}
{\sc H.~Rahmati, J.~Striuli, and Z.~Yang}, {\em Poincaré series of fiber
  products and weak complete intersection ideals}, Journal of Algebra, 498
  (2018), pp.~129 -- 152.

\end{thebibliography}

\vspace{1.25cm}

\noindent Department of Mathematics, Syracuse University, Syracuse, NY, 13244
\newline
\textit{email}: rngettin@syr.edu

 \end{document}